\documentclass[11pt,twoside,reqno,letter]{amsart}
\usepackage[margin=1in]{geometry}
\usepackage{amsmath,graphicx,varioref,amscd,amssymb,color,bm,amsthm,epsfig,color,enumerate,fancybox,bbm,esint,upref,xcolor,latexsym,mathtools,breqn,appendix,setspace,physics,tikz,pdfsync}
\usepackage[pdftex, colorlinks=true, linktocpage=true]{hyperref}

\makeatletter
\def\@tocline#1#2#3#4#5#6#7{\relax
	\ifnum #1>\c@tocdepth 
	\else
	\par \addpenalty\@secpenalty\addvspace{#2}%
	\begingroup \hyphenpenalty\@M
	\@ifempty{#4}{%
		\@tempdima\csname r@tocindent\number#1\endcsname\relax
	}{%
		\@tempdima#4\relax
	}%
	\parindent\z@ \leftskip#3\relax \advance\leftskip\@tempdima\relax
	\rightskip\@pnumwidth plus4em \parfillskip-\@pnumwidth
	#5\leavevmode\hskip-\@tempdima
	\ifcase #1
	\or\or \hskip 1em \or \hskip 2em \else \hskip 3em \fi%
	#6\nobreak\relax
	\dotfill\hbox to\@pnumwidth{\@tocpagenum{#7}}\par
	\nobreak
	\endgroup
	\fi}

\def\l@subsection{\@tocline{2}{0pt}{2pc}{5pc}{}}
\makeatother

\numberwithin{equation}{section}

\theoremstyle{plain}
\newtheorem{thm}{Theorem}[section]
\newtheorem*{thm*}{Theorem}
\newtheorem{mydef}[thm]{Definition}
\newtheorem{lem}[thm]{Lemma}
\newtheorem*{lem*}{Lemma}

\newtheorem*{prop*}{Proposition}

\theoremstyle{remark}
\newtheorem{rem}[thm]{Remark}

\newcommand{\R}{\mathbb{R}}
\newcommand{\C}{\mathbb{C}}
\newcommand{\N}{\mathbb{N}}

\newcommand{\T}{\mathbb{T}}
\newcommand{\Leray}{\mathcal{P}}

\newcommand{\mi}{m_i}
\newcommand{\Mi}{M_i}
\newcommand{\mf}{m_f}
\newcommand{\Mf}{M_f}
\newcommand{\veps}{\varepsilon}

\begin{document}
	
	
	\title[On the mass transfer in the 3D Pitaevskii model]{On the mass transfer in the 3D Pitaevskii model}

	\author[Jang]{Juhi Jang}
	\address[Jang]{\newline
		Department of Mathematics \\ University of Southern California \\ Los Angeles, CA 90089, USA}
	\email[]{\href{juhijang@usc.edu}{juhijang@usc.edu}}
	
	\author[Jayanti]{Pranava Chaitanya Jayanti}
	\address[Jayanti]{\newline
		Department of Mathematics \\ University of Southern California \\ Los Angeles, CA 90089, USA}
	\email[]{\href{pjayanti@usc.edu}{pjayanti@usc.edu}}
	
	\author[Kukavica]{Igor Kukavica}
	\address[Kukavica]{\newline
		Department of Mathematics \\ University of Southern California \\ Los Angeles, CA 90089, USA}
	\email[]{\href{kukavica@usc.edu}{kukavica@usc.edu}}

	\date{\today}
	
	
	
	
	\maketitle
	
	\begin{abstract}
		We examine a micro-scale model of superfluidity derived by Pitaevskii \cite{Pitaevskii1959PhenomenologicalPoint} which describes the interacting dynamics between superfluid He-4 and its normal fluid phase. This system consists of the nonlinear Schr\"odinger equation and the incompressible, inhomogeneous Navier-Stokes equations, coupled to each other via a bidirectional nonlinear relaxation mechanism. The coupling permits mass/momentum/energy transfer between the phases, and accounts for the conversion of superfluid into normal fluid. We prove the existence of global weak solutions in $\T^3$ for a power-type nonlinearity, beginning from small initial data. The main challenge is to control the inter-phase mass transfer in order to ensure the strict positivity of the normal fluid density, while obtaining time-independent a priori estimates.
	\end{abstract}

	\setcounter{tocdepth}{2} 
	
	
	\section{Introduction and mathematical model}  \label{intro}

    In this article, we present a rigorous analysis of the Pitaevskii model (a micro-scale description) of superfluidity ~\cite{Pitaevskii1959PhenomenologicalPoint} in three dimensions. The system consists of a superfluid phase and a normal fluid phase, described by modified versions of the nonlinear Schr\"odinger equation (NLS) and the Navier-Stokes equations (NSE), respectively. This is one of many different theories proposed to explain and quantify the underlying mechanisms of superfluidity. For more details, see ~\cite{Barenghi2001QuantizedTurbulence,Vinen2006AnTurbulence, Paoletti2011QuantumTurbulence, Berloff2014ModelingTemperatures,Jayanti2021GlobalEquations,Jayanti2022AnalysisSuperfluidity} and references therein. The Pitaevskii model, specifically, works well in the context of small length scales ($\ll$ inter-vortex spacing), and has previously been explored in ~\cite{Jayanti2022LocalSuperfluidity, Jayanti2022UniquenessSuperfluidity, Jang2023Small-dataSuperfluidity}. A similar model has also been numerically simulated in ~\cite{Brachet2023CouplingFlows}. The superfluid phase is represented by a complex-valued wavefunction $\psi$, while the normal fluid is characterized by its density $\rho$, velocity $u$, and pressure $q$. The form of the Pitaevskii model used here is as follows (with the prefix ``c'' in the equation labels signifying that the equations are coupled):
    \begin{align}
		\partial_t \psi + \lambda B\psi &= -\frac{1}{2i}\Delta\psi + \frac{\mu}{i}\lvert\psi\rvert^p \psi \tag{c-NLS} \label{NLS} \\
		B = \frac{1}{2}\left(-i\nabla - u \right)^2 + \mu \lvert \psi \rvert^p &= -\frac{1}{2}\Delta + \frac{1}{2}\lvert u \rvert^2 + iu\cdot\nabla + \mu \lvert \psi \rvert^p \tag{CPL} \label{coupling} \\
		\partial_t \rho + \nabla\cdot(\rho u) &= 2\lambda\Re(\Bar{\psi}B\psi) \tag{c-CON} \label{continuity} \\
		\partial_t (\rho u) + \nabla\cdot (\rho u \otimes u) + \nabla q - \nu \Delta u + \alpha\rho u &= 
			-2\lambda \Im(\nabla \Bar{\psi}B\psi) + \lambda\nabla \Im(\Bar{\psi}B\psi) \tag{c-NSE} \label{NSE} + \frac{\mu}{2}\nabla\lvert\psi\rvert^{p+2} \\
		\nabla\cdot u &= 0 . \tag{DIV} \label{divergence-free}
	\end{align}
	The strength of the superfluid's scattering self-interactions is measured by $\mu>0$, while $\nu>0$ and $\alpha>0$ denote the viscosity and drag coefficient of the normal fluid. The quantum scattering is a power-type nonlinearity, with exponent $p\in [1,\infty)$. The interactions between the two phases are mediated by the nonlinear coupling operator $B$, and the strength of the coupling is quantified by the positive constant $\lambda$. The structure of this coupling permits \textit{bidirectional} mass/momentum/energy transfer between the two phases, which results in a relaxation mechanism for the otherwise non-dissipative NLS. Indeed, the coupling gives ~\eqref{NLS} a parabolic flavor. These equations are supplemented with the initial conditions
	\begin{equation*} \tag{INI} \label{initial conditions}
		\psi(0,x) = \psi_0(x), \quad u(0,x) = u_0(x), \quad \rho(0,x) = \rho_0(x) \quad \text{a.e.} \ x\in\T^3 .
	\end{equation*}
	We use periodic boundary conditions, i.e., we are working in a 3-dimensional torus $[0,1]^3$. The above equations are a slight modification of Pitaevskii's original work ~\cite{Pitaevskii1959PhenomenologicalPoint}, which were valid for any type of scattering interactions and for a compressible, heat-conducting normal fluid. The simplifying assumptions used here are detailed in ~\cite{Jang2023Small-dataSuperfluidity}.

    While ~\eqref{divergence-free} implies that the velocity is divergence-free, this does not mean that the density is constant along particle trajectories. In fact, the density is governed by ~\eqref{continuity}, a transport equation with a complicated source term. This inhomogeneity on the RHS is the principal limiting factor of the analysis, since it forces us to ensure that the normal fluid density does not become zero (vacuum) or negative (unphysical). By integrating ~\eqref{continuity} over $\T^3$, the advective term vanishes and using the positivity of the operator $B$ ~\cite[Lemma 2.7]{Jang2023Small-dataSuperfluidity}, we have
	\begin{equation}
		\frac{d}{dt} \int_{\T^3}\rho \ dx = 2\lambda \Re\int_{\T^3}\Bar{\psi}B\psi \ge 0 .
	\end{equation} 
	This implies that the net mass of the normal fluid does not decrease with time. Due to conservation of total mass (of both fluids), we have
    \begin{equation}
        \frac{d}{dt} \int_{\T^3} \left(\rho + \abs{\psi}^2 \right) \ dx = 0.
    \end{equation}
    In other words, the coupling results in a conversion of superfluid into normal fluid, \textit{on average}. However, the RHS of ~\eqref{continuity} is not necessarily non-negative for all $x\in\T^3$, and we have to control the local mass transfer between the two phases. Indeed, this means bounding the $L^{\infty}$ norm (in space) of the source term, $\Re(\overline{\psi}B\psi)$, as described in Section ~\ref{the strategy}. Since the coupling is a second-order differential operator, this essentially translates into finding an estimate of $\norm{\psi}_{L^{\infty}_x}\norm{\Delta \psi}_{L^{\infty}_x}$, which dictates the required regularity of the wavefunction. In the case of no source term for the continuity equation, there are several results pertaining to the standard version of the incompressible, inhomogeneous NSE --- see ~\cite{Kazhikov1974Fluid,Ladyzhenskaya1978UniqueFluids,Kim1987WeakDensity,Simon1990NonhomogeneousPressure,Lions1996MathematicalMechanics,Danchin2003Density-dependentSpaces,Choe2003StrongFluids,Danchin2006TheFluids} among many other references.
    
    The NLS, an archetypal dispersive PDE, is often used to study quantum systems with low-energy wave interactions, like dipolar gases ~\cite{Carles2008OnGases,Sohinger2011BoundsEquations} and quantum ferrofluids ~\cite{Bland2018QuantumTurbulence}. There is a rich literature on the mathematical analysis of the NLS (see ~\cite{Tao2006NonlinearAnalysis} for a collection of results). One particular aspect of NLS that has attracted much attention is its quantum hydrodynamic (QHD) reformulation ~\cite{Jungel2002LocalEquations,Jungel2010GlobalFluids,Antonelli2017GlobalEquations,Antonelli2021GenuineStability,Wang2021AModel,Antonelli2023AnSystems}. Another connection with the compressible fluid dynamics community has been the study of Korteweg models ~\cite{Hattori1994SolutionsType,Hattori1996GlobalMaterials,Bresch2004QuelquesKorteweg,Carles2012MadelungKorteweg}, of which QHD and even capillary flows~\cite{Antonelli2022GlobalEquations} are special cases. The linear drag included in ~\eqref{NSE} is not unheard of: similar terms have been used previously to either prove the global existence of solutions ~\cite{Chauleur2022GlobalEquations}, or to show relaxation to a steady state ~\cite{Bresch2022OnTerm,Su2022ExponentialForce}. For QHD with a combination of linear drag and electrostatic forces, see~\cite{Jungel2004QuantumDecay,Antonelli2009OnDynamics,Antonelli2012TheDimensions}.

    Given the vast mathematical literature that exists on the NSE and NLS independently, it is surprising that there are hardly any on a combination of the two, like the Pitaevskii model or others ~\cite{Khalatnikov1969AbsorptionPoint,Carlson1996AVortices}. The first attempt to study a combined model was by Antonelli and Marcati ~\cite{Antonelli2015FiniteSuperfluidity}, in which a fractional time-step method pioneered by the same authors ~\cite{Antonelli2009OnDynamics,Antonelli2012OnRotation} was utilized. In this approach, the standard, uncoupled NLS is solved over a small time interval at the end of which the wavefunction is ``updated'' to account for the interactions with the normal fluid. However, this was still a uni-directionally coupled system, i.e., the wavefunction $\psi$ was dependent on the density $\rho$ and velocity $u$, but not the other way around. This yields the standard form of the continuity equation: without a source term.

    The small-data global existence of solutions for the Pitaevskii model on $\T^2$ was established in ~\cite{Jang2023Small-dataSuperfluidity} for the nonlinearity exponent $p\in[1,4)$. Moreover, for $p\ge 4$, an almost-global existence was shown, wherein the existence time grows with decreasing data size, exponentially for $p=4$ and polynomially for $p>4$. The path to the required a priori bounds of $\psi$ is by energy-type estimates, up to $L^{\infty}_t H^{\frac{5}{2}}_x$.
    
    In the current work, we obtain the global existence of weak solutions in $\T^3$, for the entire range of power nonlinearities (characterized by $1\le p<\infty$), thus proving that the inter-phase mass transfer may be controlled for small initial data. We emphasize that we deal with density that is simply in $L^{\infty}_x$, and thus not necessarily differentiable. This limits the highest regularity that one can subject the momentum equation (and the velocity) to. In turn, through the coupling operator $B$, this also restricts the regularity we can achieve for the wavefunction. The primary objective is to derive time-independent a priori bounds on $\psi$ that ensure that the RHS of ~\eqref{continuity} will not lead to non-positive densities.
    
    The approach used for 2D is not applicable in 3D, owing to the insufficiency of the $L^{\infty}_t H^1_x \cap L^2_t H^2_x$ bound for the velocity. This regularity of $u$ is not enough to derive an energy estimate of the required order for $\psi$ in 3D. Any higher regularity of $u$ would mean taking the derivatives of $\rho$, which is off limits. Another roadblock is the extremely slow mass decay for high values of $p$, i.e., when the superfluid interacts much more strongly with itself rather than the normal fluid. In such a scenario, a purely energy-based method~\cite{Jang2023Small-dataSuperfluidity} led to a bifurcation in the existence time of solutions (global or almost-global) for different ranges of the nonlinear index $p$. The mass decay rate (independent of dimension) pervades all levels of energy estimates, dominating the decay of higher norms as well. We overcome these challenges by a hybrid approach: combining the decay of superfluid mass with a maximal regularity estimate for parabolic equations (Section ~\ref{elliptic operators and maximal parabolic regularity}). The time-control of the mass conversion and of the higher order energy norm are presented in Lemmas~\ref{lem:superfluid mass estimate} and~\ref{lem:energy + higher order energy estimates}, and state that the $L^{2}_x$ and $\Dot{H}^{2}_x$ norms of the wavefunction decrease as $(1+t)^{-\frac{1}{p}}$ and $(1+t)^{-\frac{1}{2}-\frac{1}{p}}$, respectively. These allow us to evaluate the integral in ~\eqref{constraint to choose existence time}, independent of the final time $T$, leading to global control of the solution. In order to apply maximal regularity, the initial wavefunction $\psi_0$ must belong to an interpolation (Besov) space, which is carefully chosen to be marginally larger than $H^2$, so that the assumption $\psi_0\in H^2$ is sufficient for our purposes. The resulting solution $\psi$ is shown to belong to $C([0,\infty);H^2)$, among other spaces. To summarize, we demonstrate that wielding the power of parabolic regularity allows us to guarantee global solutions, even when the mass decay is exceedingly small.

    We now briefly outline the notation used in this article. Following this, we state and discuss the main result in Section ~\ref{main result}. Several a priori estimates are derived in Section ~\ref{a priori estimates}, which ends with an argument on ensuring a positive lower bound for the density. In this paper, we only present the required a priori estimates, as the general construction of solutions (and the density renormalization) follows as in ~\cite[Section 4]{Jang2023Small-dataSuperfluidity}.
	
	
	\subsection{Notation} \label{notation}
	We denote by $H^s(\T^3)$ the completion of $C^{\infty}(\T^3)$ under the Sobolev norm $H^s$. When referring to the homogeneous Sobolev spaces, we use $\Dot{H}^s(\T^3)$. Consider a 3D vector-valued function $u\equiv (u_1,u_2,u_3)\in C^{\infty}(\T^3)$. The set of all divergence-free, smooth 3D functions $u$ defines $C^{\infty}_d(\T^3)$. Then, $H^s_d(\T^3)$ is the completion of $C^{\infty}_d(\T^3)$ under the $H^s$ norm. There are many equivalent ways of defining Besov spaces, and the most appropriate choice for our purposes is through the method of real interpolation between Sobolev spaces \cite{Adams1978RealRn,Adams2003SobolevSpaces}. For $1\le q\le \infty, 0<\theta<1$, and $s = (1-\theta)s_1 + \theta s_2$, we define
	\begin{equation*}
		B^s_{2,q} := (H^{s_1},H^{s_2})_{\theta,q}.
	\end{equation*}
	
	The $L^2$ inner product, denoted by $\langle \cdot,\cdot \rangle$, is sesquilinear (the first argument is complex conjugated, indicated by an overbar) to accommodate the complex nature of the Schr\"odinger equation. Explicitly, $\langle \phi,\psi \rangle = \int_{\T^3} \Bar{\phi}\psi \ dx$.
	
	We use the subscript $x$ on a Banach space to denote that the Banach space is defined over $\T^3$. For instance, $L^r_x := L^r(\T^3)$ and $H^s_{d,x} := H^s_d(\T^3)$. For spaces/norms over time, the subscript $t$ is used, such as $L^r_t$. 
	
	We also use the notation $X\lesssim Y$ and $X\gtrsim Y$ to imply that there exists a positive constant $C$ such that $X\le CY$ and $CX\ge Y$, respectively. When appropriate, the dependence of the constant on various parameters shall be denoted using a subscript as $X\lesssim_{k_1,k_2} Y$ or $X\le C_{k_1,k_2}Y$. Throughout the article, $C$ is used to denote a (possibly large) constant that depends on the system parameters listed in ~\eqref{small data condition statement}, while $\kappa$ is used to represent a (small) positive number. The values of $C$ and $\kappa$ can vary across the different steps of calculations.

	\section{Main result and discussion} \label{main result}

    \subsection{Weak solutions and the existence theorem}
    First, we define the notion of a weak solution used here.
	\begin{mydef}[Weak solutions\footnote{See Remark ~\ref{strong or weak solutions?}.}] \label{definition of weak solutions}
		For a given time $T>0$, a triplet $(\psi,u,\rho)$ is a weak solution to the Pitaevskii model if
		\begin{enumerate} [(i)]
			\item 
			$\psi\in L^2(0,T;H^3(\T^3)), u\in L^2(0,T;H^2_d(\T^3)), \rho \in L^{\infty}([0,T]\times\T^3)$, and
			
			\item $\psi$, $u$, and $\rho$ satisfy the governing equations in the sense of distributions for all test functions, i.e.,
			\begin{equation} \label{weak solution wavefunction}
				\begin{aligned}
					&-\int_0^T \int_{\T^3} \left( \psi\partial_t\Bar{\varphi} + \frac{1}{2i}\nabla\psi\cdot\nabla\Bar{\varphi} - \lambda\Bar{\varphi}B\psi - i\mu\Bar{\varphi}\lvert \psi \rvert^p\psi \right) dx \ dt \\ 
					&\quad = \int_{\T^3} \Big( \psi_0\Bar{\varphi}(0) - \psi(T)\Bar{\varphi}(T) \Big) dx
				\end{aligned}
			\end{equation}
			and
			\begin{equation} \label{weak solution velocity}
				\begin{aligned}
					&-\int_0^T \int_{\T^3} \Big( \rho u\cdot \partial_t \Phi + \rho u\otimes u:\nabla\Phi - \nu\nabla u:\nabla\Phi - 2\lambda\Phi\cdot\Im(\nabla\Bar{\psi}B\psi) + \alpha\rho u\cdot\Phi \Big) dx \ dt \\ 
					&\quad = \int_{\T^3} \Big( \rho_0 u_0 \Phi(0) - \rho(T)u(T)\Phi(T) \Big) dx
				\end{aligned}
			\end{equation}
			and
			\begin{equation} \label{weak solution density}
				-\int_0^T \int_{\T^3} \Big( \rho \partial_t \sigma + \rho u\cdot\nabla\sigma + 2\lambda \sigma \Re(\Bar{\psi}B\psi) \Big) dx \ dt = \int_{\T^3} \Big( \rho_0 \sigma(0) - \rho(T)\sigma(T) \Big) dx
			\end{equation}
			where $\psi_0 \in H^2(\T^3)$, $u_0 \in H^1_d(\T^3)$ and $\rho_0 \in L^{\infty}(\T^3)$ are the initial data. The test functions are:
			\begin{enumerate}
				\item a complex-valued scalar field $\varphi \in H^1(0,T;L^2(\T^3))\cap L^2(0,T;H^1(\T^3))$,
				\item a real-valued, divergence-free (3D) vector field $\Phi \in H^1(0,T;L^2_d(\T^3)) \cap L^2(0,T;H^1_d(\T^3))$, and
				\item a real-valued scalar field $\sigma \in H^1(0,T;L^2(\T^3))\cap L^2(0,T;H^1(\T^3))$.
			\end{enumerate}
		\end{enumerate}
	\end{mydef}
 
	\begin{rem} \label{gradient terms in NSE}
		The last two terms in ~\eqref{NSE} are pure gradients, and thus we can absorb them into the pressure, relabeling the latter as $q$. Due to the use of the divergence-free test functions, all gradient terms in the definition of the weak solution disappear. 
	\end{rem}
	
	Now, we state the main result.
	\begin{thm} [Global existence] \label{global existence}
        Fix $p\in [1,\infty)$, and choose $0<\delta< \min\{\frac{1}{3},\frac{1}{p-1}\}$. Let $\psi_0 \in~H^2(\T^3)$, and let $u_0\in H^1_d(\T^3)$. Suppose $0< \mi\le\rho_0\le \Mi <\infty$ a.e.\ in ~$\T^3$. Then, there exists a global weak solution $(\psi,u,\rho)$ to the Pitaevskii model such that the density is always bounded between $\mf\in (0,\mi)$ and $\Mf := \Mi + \mi - \mf$, provided the initial data satisfy the smallness condition
		\begin{equation} \label{small data condition statement}
			\norm{\psi_0}_{H^2_x}^2 + \norm{u_0}_{H^1_x}^2 + \norm{\psi_0}_{L^{p+2}_x}^{p+2} \le \veps_0(\lambda,\mu,\nu,\mi,\Mi,\mf,\alpha,p) .
		\end{equation} 
        The solution has the regularity
		\begin{gather}
			\psi\in C([0,\infty);H^2(\T^3)) \cap L^2(0,\infty;H^3(\T^3)) \cap L^{1+\delta}(0,\infty;\dot{H}^{\frac{7}{2}+\delta_1}(\T^3)) \label{weak solution psi regularity} \\
			u\in C([0,\infty);H^1_d(\T^3)) \cap L^2(0,\infty;H^2(\T^3)) \label{weak solution u regularity} \\
			\rho\in L^{\infty}([0,\infty)\times \T^3)\cap C([0,\infty);L^s(\T^3)) , \label{weak solution rho regularity}
		\end{gather}
		for a sufficiently small $\delta_1>0$, and $1\le s\le 6$. Additionally, the solution satisfies the energy equality
		\begin{equation} \label{energy equality for weak solutions}
			\begin{aligned}
				& \frac{1}{2}\norm{\sqrt{\rho(t)}u(t)}_{L^2_x}^2 + \frac{1}{2}\norm{\nabla\psi(t)}_{L^2_x}^2 + \frac{2\mu}{p+2}\norm{\psi(t)}_{L^{p+2}_x}^{p+2} \\
				&\quad + \nu\norm{\nabla u}_{L^2_{[0,t]}L^2_x}^2 + \alpha\norm{\sqrt{\rho} u}_{L^2_{[0,t]}L^2_x}^2 + 2\lambda\norm{B\psi}_{L^2_{[0,t]}L^2_x}^2 \\ 
				&= \frac{1}{2}\norm{\sqrt{\rho_0}u_0}_{L^2_x}^2 + \frac{1}{2}\norm{\nabla\psi_0}_{L^2_x}^2 + \frac{2\mu}{p+2}\norm{\psi_0}_{L^{p+2}_x}^{p+2} \quad a.e. \ t\in [0,\infty) .
			\end{aligned}
		\end{equation}
	\end{thm}
    

    The proof of Theorem ~\ref{global existence} is based on a priori estimates, a semi-Galerkin scheme to construct solutions, and an adaptation of the classical renormalization procedure for the density ~\cite[Theorem 2.4]{Lions1996MathematicalMechanics}. Since the coupling operator $B$ contains the velocity $u$ by itself (and not in combination with $\rho$), we limit the calculations to when the density has a positive lower bound. This gives us an indication to the level of regularity expected of the RHS of ~\eqref{continuity}, which in turn defines the spaces in which $\psi$ and $u$ belong to. In order to achieve this, we derive the required a priori control for the wavefunction and velocity, while ensuring that the density is neither differentiated nor does it become zero anywhere in the domain.

    Before a more specific discussion on the method of proof, a few remarks about the result are warranted.
    \begin{rem} \label{restricting s<=6}
        Once we have $\rho\in L^{\infty}_{t,x}$, the renormalization procedure from ~\cite{DiPerna1989OrdinarySpaces} can be used to show that $\rho$ indeed belongs to $C_t L^s_x$. In a finite 3D domain, the Sobolev embedding $H^1 \subset L^s$ for $1\le s\le 6$ accounts for the integrability in Theorem ~\ref{global existence}. It is worth mentioning that in the analogous result in 2D, we get $1\le s< \infty$ due to the more favorable embedding.
    \end{rem}
	\begin{rem} \label{strong or weak solutions?}
		The regularity of the solutions seem to suggest that the wavefunction and velocity are strong solutions. Indeed this is true, as they are strongly continuous in their topologies. On the other hand, the density is truly a weak solution and is the reason for referring to the triplet as a weak solution. This low regularity of the density influences the nature of the calculations that follow, and in fact also prevent us from concluding uniqueness of the weak solutions. See \cite{Jayanti2022UniquenessSuperfluidity} for results akin to weak-strong uniqueness for the Pitaevskii model.
	\end{rem}
    \begin{rem} \label{mass transfer}
        In Lemma ~\ref{lem:superfluid mass estimate}, we establish that the mass of superfluid decreases with time (algebraically) and goes to $0$ as $t\rightarrow\infty$. Due to overall mass conservation, this means an increase in normal fluid mass, and an eventual conversion of all the superfluid into normal fluid. This inter-phase mass transfer is one of the underlying physical phenomena that the Pitaevskii model was designed to explain. In the macro-scale models of superfluidity (like the HVBK equations~\cite{Holm2001,Jayanti2021GlobalEquations}, the coupling between the two fluids is suggestively called \textit{mutual friction}, as it dissipates the overall energy of the system. In the micro-scale model that we are concerned with, such an energy sink can be interpreted as ``heating up'' the Bose-Einstein condensate (superfluid particles) into excited states (normal fluid particles).
    \end{rem}
    \begin{rem} \label{validity of results in 2D}
        We point out that Theorem ~\ref{global existence} is also valid for the 2D Pitaevskii model (with the same regularity, except for the density renormalization argument holding for all $1\le s<\infty$).
    \end{rem}
    \begin{rem} \label{extension to R^3}
        It would be interesting to consider the Pitaevskii model in $\R^3$. It may help to localize the dynamics by using an external confining potential (in the NLS) that rises in strength with increasing distance from the origin. This would be akin to trapped-ion quantum systems in condensed matter physics. In such a scenario, it is plausible to expect that the current results from $\T^3$ would continue to be valid in $\R^3$ as well. We thank one of the anonymous referees for posing this question.
    \end{rem}
    
	\subsection{The strategy} \label{the strategy}

    As indicated above, the main difficulty is to guarantee that if we begin from $\rho_0$ that is bounded below, then the time evolution does not result in a degeneration where $\rho$ vanishes at certain points in the domain. Hence, we define our existence time so that $\rho$ does not go below a fixed lower bound until $t=T_*$. So we aim to show that the chosen lower bound can be maintained for an arbitrarily long time.
	\begin{mydef} [Local existence time] \label{existence time definition}
		Start with an initial density field $0<\mi \le\rho_0(x)\le \Mi <\infty$. Given $0<\mf<\mi$, we define the existence time for the solution as
		\begin{equation} \label{abstract definition existence time}
			T_* := \inf \{t>0 \ \lvert \ \inf_{\T^3}\rho(t,x) = \mf\}.
		\end{equation}
	\end{mydef} 
	Consider the Lagrangian path of a particle starting at $y \in\T^3$, as it is advected by the local velocity. These characteristic curves are denoted by $X_{y}(t)$ and solve the ODE given by
	\begin{equation} \label{characteristics}
		\begin{aligned}
			\frac{d}{dt}X_{y}(t) &= u(t,X_y(t)) \\
			X_y(0) &= y \in\T^3 ,
		\end{aligned}
	\end{equation} 
	where $u$ is the velocity of the normal fluid. Traveling along such a curve, we observe that
	\begin{equation} \label{density along characteristic}
		\rho(t,X_y(t)) = \rho_0(y) + 2\lambda\Re\int_0^t \Bar{\psi}B\psi (\tau,X_y(\tau)) \ d\tau 
	\end{equation} 
	is a (formal) solution to the continuity equation. From ~\eqref{abstract definition existence time} and ~\eqref{density along characteristic}, it is clear that a sufficient condition to ensure the density is bounded from below by $\mf\in(0,\mi)$ is
	\begin{equation} \label{constraint for positive density}
		2\lambda\int_0^T \lvert\Bar{\psi}B\psi\rvert (\tau,X_y(\tau)) \ d\tau < \mi-\mf .
	\end{equation}
	This is, in turn, guaranteed by the sufficiency
	\begin{equation} \label{constraint to choose existence time}
		2\lambda \int_0^T \norm{\psi}_{L^{\infty}_x} \norm{B\psi}_{L^{\infty}_x} < \mi-\mf .
	\end{equation}
	By selecting small enough data so that all the arguments may be bootstrapped, it is possible to achieve ~\eqref{constraint to choose existence time} independently of $T>0$. Since $B\psi$ involves a second-order derivative, its $L^{\infty}_x$ norm translates into high-regularity Sobolev spaces. For $u$, this means showing that it belongs to $L^2_t H^2_x \cap H^1_t L^2_x$, which also proves useful in establishing strong continuity in time for the solution. As for the wavefunction, we make use of the parabolic nature of ~\eqref{NLS} to derive the necessary regularity (see Lemma ~\ref{maximal parabolic regularity} below). It is important to remember that throughout these calculations, we handle the density only in $L^{\infty}_x$, and not in any derivative spaces. 

    
    \subsection{Elliptic operators and maximal parabolic regularity} \label{elliptic operators and maximal parabolic regularity}
    We now define uniform ellipticity in the context of complex-valued Banach spaces, before stating the maximal parabolic regularity result that will be utilized in Lemma ~\ref{lem:maximal regularity for psi}.

    \begin{mydef}[Uniform $(K,\zeta)$-ellipticity \cite{Pruss2001SolvabilityTime}] \label{K,zeta ellipticity}
        For a complex-valued Banach space $X$, consider the differential operator
        \begin{equation} \label{differential operator}
            A(t,x) = \sum_{\abs{\alpha}=2m} a_{\alpha}(t,x) \partial^{\alpha} ,
        \end{equation}
        with domain $D(A(t,x))\subset X$, where $\alpha$ is a multi-index and $\partial^{\alpha}$ denotes spatial derivatives. The coefficients $a_{\alpha}$ are bounded and uniformly continuous functions from $[0,T]\times \R^n$ to $\C^{N\times N}$ for some $n,N\in \N$. The principal symbol associated with this operator is
        \begin{equation} \label{principal symbol differential operator}
            \Tilde{A}(t,x,\xi) = (-1)^m \sum_{\abs{\alpha}=2m} a_{\alpha}(t,x) \xi^{\alpha}.
        \end{equation}
        The operator $A(t,x)$ is said to be uniformly $(K,\zeta)$-elliptic if there exist $K\ge 1$ and $\zeta \in \left[0,\frac{\pi}{2}\right)$ such that
        \begin{enumerate}[(i)]
            \item $\sum_{\abs{\alpha}=2m} \norm{a_{\alpha}}_{L^{\infty}_{t,x}} \le K$,
            \item $\abs{\Tilde{A}(t,x,\xi)^{-1}} \le K$, and
            \item $\sigma\left(\Tilde{A}(t,x,\xi)\right) \subset \Sigma_{\zeta}\setminus \{0\}$,
        \end{enumerate}
        for $(t,x)\in [0,T]\times \R^n$, and $\xi\in \R^n$ with $\abs{\xi}=1$. Here, $\sigma(B)$ refers to the spectrum of the operator $B$, and $\Sigma_{\zeta}:= \{z\in\C : \abs{\arg{z}}\le \zeta\}$ is a sector in the right half of the complex plane.
    \end{mydef}

    It is possible to show that a uniformly $(K,\zeta)$-elliptic operator generates an analytic semigroup of negative type, leading to the maximal regularity below.
    
    \begin{lem}[Maximal parabolic regularity] \label{maximal parabolic regularity}
        Let $X$ be a (complex-valued) reflexive Banach space and $A\colon X_1\rightarrow X$ be a $(K,\zeta)$-elliptic operator defined on $D(A) = X_1 \subset X$. For $T>0$, consider the initial value problem
        \begin{equation} \label{parabolic PDE}
            \begin{gathered}
                \partial_t u(t) + Au(t) = f(t) \\
                u(0) = u_0 ,
            \end{gathered}
        \end{equation}
        where $f\in L^r([0,T];X)$ with $1<r<\infty$ and $u_0 \in X$. If it is known that $u_0$ belongs to the real interpolation space $Y:=(X,X_1)_{1-\frac{1}{r},r}$, then there exists a unique solution $u\in W^{1,r}([0,T];X) \cap L^r([0,T];X_1)$ to ~\eqref{parabolic PDE} satisfying the maximal parabolic regularity estimate
        \begin{equation}
            \norm{u}_{L^r([0,T];X)} + \norm{u}_{L^r([0,T];X_1)} + \norm{\partial_t u}_{L^r([0,T];X)} \le C_r \left( \norm{u_0}_{Y} + \norm{f}_{L^r([0,T];X)} \right).
        \end{equation}
    \end{lem}
    
    For the proof of this lemma, see ~\cite[Section 4]{Pruss2001SolvabilityTime}.
	
	

	\section{A priori estimates} \label{a priori estimates}
	
	We now derive the required a priori estimates, using formal calculations. We assume the wavefunction and velocity are smooth functions and that the density is bounded from below by $\mf>0$ in $[0,T]$. Here, $T$ is any time less than the local existence time $T_*$, and is extended to global existence in Section ~\ref{ensuring positive density}. The derivations of some estimates are identical to the 2D case ~\cite{Jang2023Small-dataSuperfluidity}, and are not repeated here.
	
	\subsection{Superfluid mass estimate}
	
	\begin{lem}[Algebraic decay rate of superfluid mass] \label{lem:superfluid mass estimate}
		The mass $S(t)$ of the superfluid decays algebraically in time as $(1+t)^{-\frac{2}{p}}$. Specifically,
        \begin{equation} \label{superfluid mass bound}
			S(t) := \norm{\psi}_{L^2_x}^2(t) \lesssim \frac{S_0}{\left(1+S_0^{\frac{p}{2}} t\right)^{\frac{2}{p}}} , \quad t \in [0,T] ,
		\end{equation} 
		where $S_0 := \norm{\psi_0}_{L^2_x}^2$ is the initial mass of the superfluid.
	\end{lem} 
    \begin{proof}
        See Lemma 3.1 of ~\cite{Jang2023Small-dataSuperfluidity}.
    \end{proof}

	\subsection{Energy estimate} \label{energy estimate}
	The energy of the system is defined as
	\begin{equation} \label{energy definition}
		E(t) := \frac{1}{2}\norm{\sqrt{\rho(t)}u(t)}_{L^2_x}^2 + \frac{1}{2}\norm{\nabla\psi(t)}_{L^2_x}^2 + \frac{2\mu}{p+2}\norm{\psi(t)}_{L^{p+2}_x}^{p+2}.
	\end{equation}
	\begin{lem} \label{energy balance}
	    The energy balance of the system is given by
        \begin{equation} \label{Energy equation}
		\frac{d}{dt}E(t) + \nu\norm{\nabla u}_{L^2_x}^2 + \alpha\norm{\sqrt{\rho} u}_{L^2_x}^2 + 2\lambda\norm{B\psi}_{L^2_x}^2 = 0.
	\end{equation}
	\end{lem}
    \begin{proof}
        See Section 3.2 of ~\cite{Jang2023Small-dataSuperfluidity}.
    \end{proof}
	Integrating ~\eqref{Energy equation} in time, we observe that the energy is bounded from above as
	\begin{equation} \label{energy bound E0}
		E(t) + \nu\norm{\nabla u}_{L^2_{[0,T]}L^2_x}^2 + \alpha\norm{\sqrt{\rho}u}_{L^2_{[0,T]}L^2_x}^2 + 2\lambda\norm{B\psi}_{L^2_{[0,T]}L^2_x}^2 \\ = E_0 , \quad t\in [0,T] ,
	\end{equation}  
	where
	\begin{equation} \label{E0 definition}
		E_0 := \frac{1}{2}\norm{\sqrt{\rho_0}u_0}_{L^2_x}^2 + \frac{1}{2}\norm{\nabla\psi_0}_{L^2_x}^2 + \frac{2\mu}{p+2}\norm{\psi_0}_{L^{p+2}_x}^{p+2}
	\end{equation}
	denotes the initial energy of the system. Next, we show that the energy is not just bounded, but also decays (algebraically) with time. In order to achieve this, we observe that $\norm{B\psi}_{L^2_x}$ can be rewritten as $\norm{D^2\psi}_{L^2_x}$, at the expense of some nonlinear terms on the RHS. More precisely (see equation (3.14) in ~\cite{Jang2023Small-dataSuperfluidity} for the exact derivation),
	\begin{equation*} \label{B psi is bounded from below by D^2 psi}
		\begin{aligned}
			\norm{B\psi}_{L^2_x}^2 \ge \frac{1}{8}\norm{D^2\psi}_{L^2_x}^2 - C\norm{\abs{u}^2\psi}_{L^2_x}^2 - C\norm{u\cdot\nabla\psi}_{L^2_x}^2 + \frac{1}{C}\norm{\psi}_{L^{2p+2}_x}^{2p+2} + \frac{1}{C} \norm{\nabla\abs{\psi}^{\frac{p}{2}+1}}_{L^2_x}^2 .
		\end{aligned}
	\end{equation*}  
	Thus, ~\eqref{Energy equation} now becomes
	\begin{equation} \label{energy inequation 1}
		\begin{aligned}
			&\frac{dE}{dt} + \nu\norm{\nabla u}_{L^2_x}^2 + \alpha\norm{\sqrt{\rho} u}_{L^2_x}^2 + \frac{\lambda}{4}\norm{D^2\psi}_{L^2_x}^2 + \frac{1}{C}\norm{\psi}_{L^{2p+2}_x}^{2p+2} + \frac{1}{C}\norm{\nabla\abs{\psi}^{\frac{p}{2}+1}}_{L^2_x}^2 \\ 
			&\lesssim \norm{\abs{u}^2\psi}_{L^2_x}^2 + \norm{u\cdot\nabla\psi}_{L^2_x}^2 \\
			&=: I_1 + I_2.
		\end{aligned}
	\end{equation}
	The first term on the RHS is estimated as
	\begin{equation*}
		I_1 \lesssim \norm{u}_{L^6_x}^4\norm{\psi}_{L^6_x}^2 \lesssim \norm{u}_{H^1_x}^4\norm{\psi}_{H^1_x}^2
	\end{equation*}
	using H\"older's inequality and Sobolev embedding. For the second term in ~\eqref{energy inequation 1}, we interpolate the $L^3_x$ norm, and apply the Poincar\'e, H\"older's, and Young's inequalities, as well as Sobolev embedding to get
	\begin{align*}
		I_2 \lesssim \norm{u}_{L^6_x}^2\norm{\nabla\psi}_{L^3_x}^2 \lesssim \norm{u}_{H^1_x}^2\norm{\nabla\psi}_{L^2_x}\norm{D^2\psi}_{L^2_x}
		\le C_{\kappa}\norm{u}_{H^1_x}^4\norm{\nabla\psi}_{L^2_x}^2 + \kappa\norm{D^2\psi}_{L^2_x}^2 .
	\end{align*}  
	We also use the Poincar\'e inequality to convert the last term on the LHS of ~\eqref{energy inequation 1} into a coercive term for the internal energy term $\frac{2\mu}{p+2}\norm{\psi}_{L^{p+2}_x}^{p+2}$ in $E(t)$. To this end, we observe that
	\begin{equation} \label{poincare for potential energy}
		\begin{aligned}
			\norm{\psi}_{L^{p+2}_x}^{p+2} &\le \norm{\abs{\psi}^{\frac{p}{2}+1} - \frac{1}{\abs{\T^3}}\int_{\T^3} \abs{\psi}^{\frac{p}{2}+1}}_{L^2_x}^2 + \norm{\frac{1}{\abs{\T^3}}\int_{\T^3} \abs{\psi}^{\frac{p}{2}+1}}_{L^2_x}^2 \lesssim \norm{\nabla\abs{\psi}^{\frac{p}{2}+1}}_{L^2_x}^2 + \norm{\psi}_{L^{\frac{p}{2}+1}}^{p+2} \\ &\le C\norm{\nabla\abs{\psi}^{\frac{p}{2}+1}}_{L^2_x}^2 + \kappa \norm{\psi}_{L^{p+2}_x}^{p+2} + C_{\kappa} \norm{\psi}_{L^2_x}^{p+2}.
		\end{aligned}
	\end{equation}
	In the last inequality, we interpolated between the $L^{p+2}_x$ and $L^2_x$ norms, which is valid for $p>2$. For a sufficiently small $\kappa$, the second term on the RHS is absorbed into the LHS. On the other hand, when $1\le p\le 2$, the finite size of the domain implies $L^2_x \subseteq L^{\frac{p}{2}+1}_x$, which again leads to~\eqref{poincare for potential energy}. Thus, for any $p\ge 1$, ~\eqref{energy inequation 1} becomes
	\begin{equation} \label{energy inequation 3}
    \begin{aligned}
		&\frac{dE}{dt} + \nu\norm{\nabla u}_{L^2_x}^2 + \alpha\norm{\sqrt{\rho} u}_{L^2_x}^2 + \frac{1}{C}\norm{D^2\psi}_{L^2_x}^2 + \frac{1}{C}\norm{\psi}_{L^{p+2}_x}^{p+2} + \frac{1}{C}\norm{\psi}_{L^{2p+2}_x}^{2p+2} \\
		&\le C\norm{\psi}_{L^2_x}^{p+2} +  C\norm{u}_{H^1_x}^4\norm{\psi}_{H^1_x}^2.
    \end{aligned}
	\end{equation}  
	In order to show a decaying norm, we need coercive terms on the LHS, which have been achieved. To control the RHS (particularly the second term), we derive a balance equation for a higher-order energy $X(t)$, defined in~\eqref{eq:defining X}. Combining $E(t)$ with $X(t)$ allows us to close the estimates.
	
	\subsection{Higher-order energy estimate} \label{higher order estimate}
	We now present more a priori bounds for $\psi$ and $u$, involving one more derivative than the energy $E$.
	
	\subsubsection{The Schr\"odinger equation} \label{NLS higher order estimate}
	Acting upon ~\eqref{NLS} with the Laplacian $-\Delta$, multiplying by $-\Delta\Bar{\psi}$, taking the real part, and integrating over the domain yields
	\begin{equation} \label{schrodinger equation higher order first step}
		\begin{aligned}
			\frac{1}{2}\frac{d}{dt}\norm{\Delta\psi}_{L^2_x}^2 &= - \lambda\Re\int_{\T^3}(\Delta^2\Bar{\psi})B\psi + \mu\Im\int_{\T^3}(\Delta^2\Bar{\psi})\abs{\psi}^p\psi \\
			&=: I_3 + I_4.
		\end{aligned}
	\end{equation} 
	The first term on the RHS of ~\eqref{NLS} leads to a term which vanishes due to the periodic boundary conditions. We now estimate the RHS of ~\eqref{schrodinger equation higher order first step}. For the first term, we have
	\begin{align*}
		I_3 &= \lambda\Re\int_{\T^3}\nabla(\Delta\Bar{\psi})\cdot\nabla\left(-\frac{1}{2}\Delta\psi + \frac{1}{2}\abs{u}^2\psi + iu\cdot\nabla\psi +\mu\abs{\psi}^p\psi\right) \\
		&= -\frac{\lambda}{2}\norm{D^3\psi}_{L^2_x}^2 + \lambda\Re\int_{\T^3}\nabla(\Delta\Bar{\psi})\cdot\nabla\left(\frac{1}{2}\abs{u}^2\psi + iu\cdot\nabla\psi +\mu\abs{\psi}^p\psi\right) \\
		&\le -\frac{\lambda}{4}\norm{D^3\psi}_{L^2_x}^2 + C\norm{\nabla(\abs{u}^2\psi)}_{L^2_x}^2 + C\norm{\nabla(u\cdot\nabla\psi)}_{L^2_x}^2 + C\norm{\nabla(\abs{\psi}^p\psi)}_{L^2_x}^2.
	\end{align*}
	The first term on the RHS acts as the dissipative term for $\psi$ at this higher-order energy level. For $I_4$, we integrate by parts and use H\"older's inequality to obtain
	\begin{equation*}
		I_4 = -\mu\Im\int_{\T^3}\nabla(\Delta\Bar{\psi})\cdot\nabla(\abs{\psi}^p\psi) \le \frac{\lambda}{8}\norm{D^3\psi}_{L^2_x}^2 + C\norm{\nabla(\abs{\psi}^p\psi)}_{L^2_x}^2.
	\end{equation*}
	Thus, ~\eqref{schrodinger equation higher order first step} becomes
	\begin{equation} \label{schrodinger equation higher order third step}
		\begin{aligned}
			\frac{d}{dt}\norm{\Delta\psi}_{L^2_x}^2 + \frac{1}{C}\norm{D^3\psi}_{L^2_x}^2 &\lesssim  \norm{\nabla\left(\abs{u}^2\psi \right)}_{L^2_x}^2 + \norm{\nabla(u\cdot\nabla\psi)}_{L^2_x}^2 + \norm{\nabla \left( \abs{\psi}^p\psi \right) }_{L^2_x}^2 \\
			&=: I_5 + I_6 + I_7.
		\end{aligned}
	\end{equation}
	The first term is bounded using the Poincar\'e inequality, Sobolev embedding, and Lebesgue interpolation as
	\begin{align*}
		I_5 &\lesssim  \norm{u}_{L^6_x}^2\norm{\nabla u}_{L^3_x}^2\norm{\psi}_{L^{\infty}_x}^2 + \norm{u}_{L^6_x}^4\norm{\nabla\psi}_{L^6_x}^2 \\
		&\lesssim \norm{u}_{H^1_x}^2\norm{\nabla u}_{L^2_x}\norm{\Delta u}_{L^2_x}\norm{\psi}_{H^2_x}^2 + \norm{u}_{H^1_x}^4 \norm{\Delta\psi}_{L^2_x}^2 \\
		&\le C_{\kappa}\norm{u}_{H^1_x}^4 \norm{\nabla u}_{L^2_x}^2 \norm{\psi}_{H^2_x}^4 + \kappa\norm{\Delta u}_{L^2_x}^2 + C\norm{u}_{H^1_x}^4 \norm{\Delta\psi}_{L^2_x}^2.
	\end{align*}
	We applied Young's inequality to extract out dissipative terms in the last step. Again, $\kappa$ denotes a small number whose value shall be fixed later on, and $C_{\kappa}$ is a constant whose value depends on $\kappa$ and the system parameters. For the second term on the RHS of ~\eqref{schrodinger equation higher order third step}, we have
	\begin{align*}
		I_6 &\lesssim \norm{\nabla u}_{L^3_x}^2 \norm{\nabla\psi}_{L^6_x}^2 + \norm{u}_{L^6_x}^2 \norm{D^2\psi}_{L^3_x}^2 \\
		&\lesssim \norm{\nabla u}_{L^2_x}\norm{\Delta u}_{L^2_x} \norm{\Delta\psi}_{L^2_x}^2 + \norm{u}_{H^1_x}^2 \norm{\Delta\psi}_{L^2_x} \norm{D^3\psi}_{L^2_x} \\
		&\le C_{\kappa}\norm{\nabla u}_{L^2_x}^2 \norm{\Delta\psi}_{L^2_x}^4 + \kappa\norm{\Delta u}_{L^2_x}^2 + C_{\kappa}\norm{u}_{H^1_x}^4 \norm{\Delta\psi}_{L^2_x}^2 + \kappa\norm{D^3\psi}_{L^2_x}^2
	\end{align*} 
	Finally, we apply the Poincar\'e inequality and Sobolev embedding to bound $I_7$. This leads to
	\begin{equation} \label{nonlinear term p<=2}
		\begin{aligned}
			I_7 &\lesssim \norm{\abs{\psi}^p \abs{\nabla\psi}}_{L^2_x}^2 \lesssim \norm{\psi}_{L^{\infty}_x}^{2p}\norm{\nabla\psi}_{L^2_x}^2 \lesssim \left(\norm{\psi}_{L^2_x}^{2p} + \norm{\Delta\psi}_{L^2_x}^{2p}\right) \norm{\nabla\psi}_{L^2_x}^2 \\ 
            &\lesssim \norm{\psi}_{L^2_x}^{2p} \norm{\Delta\psi}_{L^2_x}^2 +  \norm{\Delta\psi}_{L^2_x}^{2p+2} .
		\end{aligned}
	\end{equation}
	
	Combining all these inequalities into ~\eqref{schrodinger equation higher order third step}, and absorbing $\kappa\norm{D^3\psi}_{L^2_x}^2$ into the LHS, we end up with
	\begin{equation} \label{schrodinger equation higher order fourth step}
		\begin{aligned}
			&\frac{d}{dt}\norm{\Delta\psi}_{L^2_x}^2 + \frac{1}{C}\norm{D^3\psi}_{L^2_x}^2 \\ 
			&\le C_{\kappa}\left(\norm{\nabla u}_{L^2_x}^2\norm{\Delta \psi}_{L^2_x}^4 + \norm{u}_{H^1_x}^4 \norm{\nabla u}_{L^2_x}^2  \norm{\psi}_{H^2_x}^4 \right) 
		    + C\norm{u}_{H^1_x}^4\norm{\Delta \psi}_{L^2_x}^2 + C\norm{\psi}_{L^2_x}^{2p} \norm{\Delta\psi}_{L^2_x}^2 \\ 
			&\quad + C\norm{\Delta\psi}_{L^2_x}^{2p+2} + \kappa\norm{\Delta u}_{L^2_x}^2.
		\end{aligned}
	\end{equation}
	This constitutes the higher-order estimate for the wavefunction. Next, we combine this with corresponding estimates for the velocity.
	
	\subsubsection{The Navier-Stokes equation} \label{NSE higher order estimate}
	We begin by rewriting ~\eqref{NSE} in the \textit{non-conservative form}, and applying the Leray projector (see Remark ~\ref{gradient terms in NSE}) to get
	\begin{equation}  \label{NSE'}
		\Leray\left(\rho\partial_t u + \rho u\cdot\nabla u - \nu \Delta u + \alpha\rho u\right) = \Leray\left(-2\lambda \Im(\nabla \Bar{\psi}B\psi) - 2\lambda u\Re(\Bar{\psi}B\psi) \right). \tag{c-NSE'}
	\end{equation} 
	Here, $\Leray$ is the Leray projector, which projects a Hilbert space into its divergence-free subspace, thus removing any purely gradient terms. Next, we multiply ~\eqref{NSE'} by $\partial_t u$ and integrate over the domain. This leads to
	\begin{equation} \label{NSE higher order first step}
		\begin{aligned}
			\int_{\T^3}\rho\abs{\partial_t u}^2 + \frac{\nu}{2}\frac{d}{dt}\norm{\nabla u}_{L^2_x}^2 &= -\int_{\T^3}\rho u\cdot\nabla u\cdot\partial_t u - 2\lambda\int_{\T^3}\partial_t u\cdot\Im(\nabla\Bar{\psi}B\psi) \\ 
			&\qquad - 2\lambda\int_{\T^3}\partial_t u\cdot u\Re(\Bar{\psi}B\psi) -\alpha\int_{\T^3}\rho u\cdot\partial_t u \\
			&=: I_8 + I_9 + I_{10} + I_{11}.
		\end{aligned}
	\end{equation} 
	Henceforth, we repeatedly use the fact that the density is bounded both above and below ($\mf\le \rho \le \Mf=\Mi+\mi-\mf$) to control the RHS. In particular, $\norm{u}_{L^2_x}$ and $\norm{\partial_t u}_{L^2_x}$ are equivalent to $\norm{\sqrt{\rho}u}_{L^2_x}$ and $\norm{\sqrt{\rho}\partial_t u}_{L^2_x}$, respectively. Thus, for the first term,
	\begin{align*}
		I_8 &\le \frac{1}{8}\norm{\sqrt{\rho}\partial_t u}_{L^2_x}^2 + C \int_{\T^3}\abs{u}^2\abs{\nabla u}^2 \le \frac{1}{8}\norm{\sqrt{\rho}\partial_t u}_{L^2_x}^2 + C\norm{u}_{L^6_x}^2 \norm{\nabla u}_{L^3_x}^2 \\
		&\le \frac{1}{8}\norm{\sqrt{\rho}\partial_t u}_{L^2_x}^2 + C_{\kappa}\norm{u}_{H^1_x}^4 \norm{\nabla u}_{L^2_x}^2 + \kappa\norm{\Delta u}_{L^2_x}^2.
	\end{align*}
	In going from the second line to the third, we use the Sobolev embeddings and Lebesgue interpolation. Finally, Young's inequality lets us extract the required dissipative term. For the second integral of ~\eqref{NSE higher order first step}, we have
	\begin{align*}
		I_9 &\le \frac{1}{8}\norm{\sqrt{\rho}\partial_t u}_{L^2_x}^2 + C\norm{\nabla\psi}_{L^6_x}^2\norm{B\psi}_{L^3_x}^2 \\
		&\le \frac{1}{8}\norm{\sqrt{\rho}\partial_t u}_{L^2_x}^2 + C\norm{\Delta\psi}_{L^2_x}^2\norm{B\psi}_{L^2_x}\norm{B\psi}_{H^1_x} \\
		&\le \frac{1}{8}\norm{\sqrt{\rho}\partial_t u}_{L^2_x}^2 + C_{\kappa}\norm{B\psi}_{L^2_x}^2 \left( \norm{\Delta\psi}_{L^2_x}^4 + \norm{\Delta\psi}_{L^2_x}^2 \right) + \kappa \norm{\nabla(B\psi)}_{L^2_x}^2
	\end{align*}  
	where the $B\psi$ term is handled via interpolation and Young's inequality, while the term $\nabla\psi$ is bounded using Sobolev embedding. For the third integral,
	\begin{align*}
		I_{10} &\le \frac{1}{8}\norm{\sqrt{\rho}\partial_t u}_{L^2_x}^2 + C\norm{u}_{L^6_x}^2\norm{\psi}_{L^{\infty}_x}^2\norm{B\psi}_{L^3_x}^2 \\
		&\le \frac{1}{8}\norm{\sqrt{\rho}\partial_t u}_{L^2_x}^2 + C_{\kappa}\norm{B\psi}_{L^2_x}^2 \left( \norm{u}_{H^1_x}^2\norm{\psi}_{H^2_x}^2 + \norm{u}_{H^1_x}^4\norm{\psi}_{H^2_x}^4 \right) + \kappa\norm{\nabla(B\psi)}_{L^2_x}^2
	\end{align*} 
	where the $B\psi$ term is handled just like in $I_9$. Finally, for the last term, we integrate by parts and use ~\eqref{continuity}, which results in
	\begin{equation} \label{HE9 3 terms}
			I_{11} 
			= -\frac{\alpha}{2} \frac{d}{dt} \norm{ \sqrt{\rho} u}_{L^2_x}^2 + \frac{\alpha}{2} \int_{\T^3} \rho u\cdot \nabla \abs{u}^2 + \alpha\lambda\int_{\T^3} \Re(\overline{\psi}B\psi) \abs{u}^2.
	\end{equation}
	We estimate the second term in ~\eqref{HE9 3 terms} via interpolation and a Sobolev embedding. This gives
	\begin{align*}
		\frac{\alpha}{2} \int_{\T^3} \rho u\cdot \nabla \abs{u}^2 &\lesssim \norm{u}_{L^2_x} \norm{u}_{L^3_x} \norm{\nabla u}_{L^6_x} \\
		&\le C_{\kappa} \norm{u}_{L^2_x}^3 \norm{u}_{L^6_x} + \kappa\norm{\Delta u}_{L^2_x}^2 \le C_{\kappa}\norm{u}_{H^1_x}^4 + \kappa\norm{\Delta u}_{L^2_x}^2.
	\end{align*}
	Similarly, for the third term in ~\eqref{HE9 3 terms},
	\begin{align*}
		\alpha\lambda\int_{\T^3} \Re(\overline{\psi}B\psi) \abs{u}^2 &\lesssim \norm{\psi}_{L^6_x} \norm{u}_{L^6_x}^2 \norm{B\psi}_{L^2_x} \le C_{\kappa} \norm{\psi}_{H^1_x}^2 \norm{u}_{H^1_x}^4 + \kappa\norm{B\psi}_{L^2_x}^2.
	\end{align*}
	Putting together the above estimates into ~\eqref{NSE higher order first step}, we end up with
	\begin{equation} \label{NSE higher order second step}
		\begin{aligned}
			&\nu\frac{d}{dt}\norm{\nabla u}_{L^2_x}^2 + \norm{\sqrt{\rho}\partial_t u}_{L^2_x}^2 + \alpha \frac{d}{dt} \norm{ \sqrt{\rho} u}_{L^2_x}^2 \\ 
			&\le
			C_{\kappa} \left( \norm{u}_{H^1_x}^4 \norm{\nabla u}_{L^2_x}^2 +  \norm{u}_{H^1_x}^4 + \norm{\psi}_{H^1_x}^2 \norm{u}_{H^1_x}^4 \right) \\ 
			&\quad + C_{\kappa}\norm{B\psi}_{L^2_x}^2 \left(\norm{\Delta\psi}_{L^2_x}^2 + \norm{u}_{H^1_x}^2\norm{\psi}_{H^2_x}^2 + \norm{\Delta\psi}_{L^2_x}^4 + \norm{u}_{H^1_x}^4\norm{\psi}_{H^2_x}^4 \right) \\ 
			&\quad + \kappa\norm{B\psi}_{L^2_x}^2 + \kappa\norm{\nabla(B\psi)}_{L^2_x}^2 + \kappa\norm{\Delta u}_{L^2_x}^2
		\end{aligned}
	\end{equation}
	where $C_{\kappa}$ depends on $\kappa$ as well as the system parameters.
	
	In order to obtain the higher-order velocity dissipation $\norm{\Delta u}_{L^2_x}^2$, we multiply ~\eqref{NSE'} by $-\theta\Delta u$, with $\theta>0$ to be fixed shortly, and integrate over the domain. This gives
	\begin{equation} \label{NSE higher order third step}
		\begin{aligned}
			\theta\nu\norm{\Delta u}_{L^2_x}^2 &= \theta\int_{\T^3}\rho\partial_t u\cdot\Delta u + \theta\int_{\T^3} \left(\rho u\cdot\nabla u \right)\cdot\Delta u + 2\lambda\theta\int_{\T^3}\Im(\nabla\Bar{\psi}B\psi)\cdot\Delta u \\ 
			&\qquad + 2\lambda\theta\int_{\T^3}u\Re(\Bar{\psi}B\psi)\cdot\Delta u + \alpha\theta\int_{\T^3} \rho u\cdot \Delta u \\
			&=: I_{12} + I_{13} + I_{14} + I_{15} + I_{16}.
		\end{aligned}
	\end{equation}
	For the first term, we have
	\begin{equation*}
		I_{12} \le \frac{\theta\nu}{10}\norm{\Delta u}_{L^2_x}^2 + C\theta\norm{\sqrt{\rho}\partial_t u}_{L^2_x}^2.
	\end{equation*} 
	The second integral is manipulated just as $I_8$, namely,
	\begin{align*}
		I_{13} \le \frac{\theta\nu}{20}\norm{\Delta u}_{L^2_x}^2 + C_{\theta}\int_{\T^3}\abs{u}^2\abs{\nabla u}^2 \le \frac{\theta\nu}{10}\norm{\Delta u}_{L^2_x}^2 + C_{\theta} \norm{u}_{H^1_x}^4 \norm{\nabla u}_{L^2_x}^2.
	\end{align*} 
	The integral $I_{14}$ requires Sobolev embedding, the Poincar\'e inequality, and Lebesgue norm interpolation, and reads
	\begin{align*}
		I_{14} &\le \frac{\theta\nu}{10}\norm{\Delta u}_{L^2_x}^2 + C_{\theta}\norm{\nabla\psi}_{L^6_x}^2\norm{B\psi}_{L^3_x}^2 \\
		&\le \frac{\theta\nu}{10}\norm{\Delta u}_{L^2_x}^2 + C_{\theta}\norm{\Delta\psi}_{L^2_x}^2\norm{B\psi}_{L^2_x} \norm{B\psi}_{H^1_x} \\
		&\le \frac{\theta\nu}{10}\norm{\Delta u}_{L^2_x}^2 + C_{\kappa,\theta}\norm{B\psi}_{L^2_x}^2 \left( \norm{\Delta\psi}_{L^2_x}^2 + \norm{\Delta\psi}_{L^2_x}^4 \right) + \kappa\norm{\nabla(B\psi)}_{L^2_x}^2.
	\end{align*} 
	In a similar manner, we have
	\begin{align*}
		I_{15} &\le \frac{\theta\nu}{10}\norm{\Delta u}_{L^2_x}^2 + C_{\theta}\norm{u}_{L^6_x}^2\norm{\psi}_{L^{\infty}_x}^2\norm{B\psi}_{L^3_x}^2 \\
		&\le \frac{\theta\nu}{10}\norm{\Delta u}_{L^2_x}^2 + C_{\theta}\norm{u}_{H^1_x}^2\norm{\psi}_{H^2_x}^2\norm{B\psi}_{L^2_x}\norm{B\psi}_{H^1_x} \\
		&\le \frac{\theta\nu}{10}\norm{\Delta u}_{L^2_x}^2 + C_{\kappa,\theta}\norm{B\psi}_{L^2_x}^2 \left( \norm{u}_{H^1_x}^2\norm{\psi}_{H^2_x}^2 + \norm{u}_{H^1_x}^4\norm{\psi}_{H^2_x}^4 \right) + \kappa\norm{\nabla(B\psi)}_{L^2_x}^2.
	\end{align*} 
	The last integral in ~\eqref{NSE higher order third step} requires, like $I_{12}$, only H\"older's and Young's inequalities. This provides
	\begin{equation*}
		I_{16} \le \frac{\theta\nu}{10}\norm{\Delta u}_{L^2_x}^2 + C\theta\norm{\sqrt{\rho}u}_{L^2_x}^2.
	\end{equation*}
	In the end, ~\eqref{NSE higher order third step} becomes
	\begin{equation} \label{NSE higher order fourth step}
		\begin{aligned}
			\frac{\theta\nu}{2}\norm{\Delta u}_{L^2_x}^2 &\le C\theta\norm{\sqrt{\rho}\partial_t u}_{L^2_x}^2 + C\theta\norm{\sqrt{\rho}u}_{L^2_x}^2 +  C_{\theta}\norm{u}_{H^1_x}^4 \norm{\nabla u}_{L^2_x}^2 \\ 
			&\quad + C_{\kappa,\theta} \norm{B\psi}_{L^2_x}^2 \left( \norm{\Delta\psi}_{L^2_x}^2 + \norm{u}_{H^1_x}^2\norm{\psi}_{H^2_x}^2 + \norm{\Delta\psi}_{L^2_x}^4 + \norm{u}_{H^1_x}^4 \norm{\psi}_{H^2_x}^4 \right) \\
			&\quad + \kappa\norm{\nabla(B\psi)}_{L^2_x}^2 .
		\end{aligned}
	\end{equation} 
	
	We now add ~\eqref{schrodinger equation higher order fourth step}, ~\eqref{NSE higher order second step} and ~\eqref{NSE higher order fourth step}. Then, we note from the definition of $B\psi$ that
	\begin{equation*}
		\norm{\nabla(B\psi)}_{L^2_x}^2 \lesssim \norm{D^3 \psi}_{L^2_x}^2 + \norm{\nabla(\abs{u}^2 \psi)}_{L^2_x}^2 + \norm{\nabla(u\cdot\nabla\psi)}_{L^2_x}^2 + \norm{\nabla(\abs{\psi}^p \psi)}_{L^2_x}^2 , 
	\end{equation*}
	where the last three terms on the RHS are exactly $I_5$, $I_6$, and $I_7$. Using sufficiently small values for $\theta$ and $\kappa$, we also absorb $\norm{\sqrt{\rho}\partial_t u}_{L^2_x}^2$ and $\norm{\Delta u}_{L^2_x}^2$ on the RHS into the LHS. Finally, we are left with
	\begin{equation} \label{delta psi nabla u equations combined}
		\begin{aligned}
			&\frac{d}{dt}\left[ \norm{\Delta\psi}_{L^2_x}^2 + \nu\norm{\nabla u}_{L^2_x}^2 + \alpha\norm{\sqrt{\rho}u}_{L^2_x}^2 \right] + \frac{1}{C}\norm{D^3\psi}_{L^2_x}^2 + \frac{1}{C}\norm{\sqrt{\rho}\partial_t u}_{L^2_x}^2 + \frac{1}{C}\norm{\Delta u}_{L^2_x}^2 \\
			&\le C\Bigg( \left(1 + \norm{u}_{H^1_x}^4 \right) \norm{\nabla u}_{L^2_x}^2 \norm{\psi}_{H^2_x}^4 + \norm{u}_{H^1_x}^4 \norm{\Delta \psi}_{L^2_x}^2 + \left(1 + \norm{\psi}_{H^1_x}^2 \right) \norm{u}_{H^1_x}^4 \Bigg) \\ 
			&\quad + C \Big(\norm{\psi}_{L^2_x}^{2p} \norm{\Delta\psi}_{L^2_x}^2 + \norm{\Delta\psi}_{L^2_x}^{2p+2} +  \norm{u}_{H^1_x}^4 \norm{\nabla u}_{L^2_x}^2 \Big) \\
			&\quad + C \norm{B\psi}_{L^2_x}^2 \left( \norm{\Delta\psi}_{L^2_x}^2 + \norm{u}_{H^1_x}^2 \norm{\psi}_{H^2_x}^2 + \norm{\Delta\psi}_{L^2_x}^4 + \norm{u}_{H^1_x}^4 \norm{\psi}_{H^2_x}^4 \right) \\
			&\quad  + C\theta\norm{\sqrt{\rho}u}_{L^2_x}^2 + \kappa\norm{B\psi}_{L^2_x}^2 .
		\end{aligned}
	\end{equation}
	This is the higher-order energy estimate. Using similar arguments to Section 3.3.3 of ~\cite{Jang2023Small-dataSuperfluidity}, we can use the Gr\"onwall inequality to control the higher-order energy and dissipation. The results are summarized in the following lemma.
	\begin{lem}[Algebraic decay rate for energies] \label{lem:energy + higher order energy estimates}
		We label the higher-order energy as
        \begin{equation} \label{eq:defining X}
            X := \norm{\Delta\psi(t)}_{L^2_x}^2 + \nu\norm{\nabla u(t)}_{L^2_x}^2 .
        \end{equation} 
        Then, the sum $Z:=X+E$ decays as 
        \begin{equation} \label{Z solution bound}
			Z(t) \le Z_0 e^{-\frac{t}{C}} + \frac{CS_0^{\frac{p}{2}+1}}{\left( 1+S_0^{\frac{p}{2}}t \right)^{1+\frac{2}{p}}} \lesssim Z_0 + S_0^{\frac{p}{2}+1} ,
		\end{equation}
        where $Z_0 := Z(0)$. Moreover, the time-integral of the corresponding dissipation terms is also bounded. Specifically,
        \begin{equation} \label{Y dissipative estimates}
			\begin{aligned} 
				&\norm{D^3 \psi}_{L^2_{[0,T]} L^2_x}^2 + \norm{\sqrt{\rho}\partial_t u}_{L^2_{[0,T]} L^2_x}^2 + \norm{\Delta u}_{L^2_{[0,T]} L^2_x}^2 + \norm{\nabla u}_{L^2_{[0,T]} L^2_x}^2 + \norm{\sqrt{\rho}u}_{L^2_{[0,T]} L^2_x}^2 + \norm{B\psi}_{L^2_{[0,T]}
					L^2_x}^2 \\ 
				&\lesssim Z_0 + S_0^p(Z_0 + S_0) \lesssim Z_0 + S_0^{p+1} .
			\end{aligned}
		\end{equation}
        In addition, by integrating the higher-order energy estimate over $[t,2t]$ (for $t\ge 1$), we end up with a time-decaying estimate for the dissipation, given by
		\begin{equation} \label{improved Y dissipative estimates}
			\begin{aligned} 
				&\norm{D^3 \psi}_{L^2_{[t,2t]} L^2_x}^2 + \norm{\sqrt{\rho}\partial_t u}_{L^2_{[t,2t]} L^2_x}^2 + \norm{\Delta u}_{L^2_{[t,2t]} L^2_x}^2 + \norm{\nabla u}_{L^2_{[t,2t]} L^2_x}^2 + \norm{\sqrt{\rho}u}_{L^2_{[t,2t]} L^2_x}^2 + \norm{B\psi}_{L^2_{[t,2t]} L^2_x}^2 \\ 
				&\lesssim Z_0 e^{-\frac{t}{C}} + \frac{S_0^{\frac{p}{2}+1}}{\left(1+S_0^{\frac{p}{2}}t\right)^{\frac{2}{p}}}.
			\end{aligned}
		\end{equation}
	\end{lem}
    The last inequality in ~\eqref{Y dissipative estimates} is valid because $Z_0,S_0\le 1$.

	\subsection{Maximal parabolic regularity for $\psi$} \label{maximal parabolic regularity for psi}
	
	From the previous analysis, we have obtained $\psi\in L^2_{[0,T]}H^3_x$. However, as pointed out in the discussion following Definition ~\ref{existence time definition}, we seek $B\psi\in L^2_{[0,T]}L^{\infty}_x$, which follows from $\psi\in L^2_{[0,T]}H^{\frac{7}{2}+}_x$. In the 2D case \cite{Jang2023Small-dataSuperfluidity}, this was achieved by taking advantage of the Sobolev embedding $H^{1+\delta}_x \subset L^{\infty}_x$ and deriving a ``highest-order energy estimate'' for $\psi$. This approach does not work here due to the embedding $H^{\frac{3}{2}+\delta}_x \subset L^{\infty}_x$, which would require higher-order estimates on $u$ and $\rho$. Instead, we exploit the parabolic nature of ~\eqref{NLS} and apply the method of maximal regularity to gain the necessary control of $\psi$. 

    \begin{lem}[Maximal regularity for $\psi$] \label{lem:maximal regularity for psi}
        For $\delta$ as defined in Theorem ~\ref{global existence}, and a sufficiently small $\delta_1>0$, we have the maximal regularity bound 
        \begin{equation} \label{maximal regularity lemma for psi}           \norm{\partial_t\psi}_{L^{1+\delta}_{[0,T]} H^{\frac{3}{2}+\delta_1}_x} + \norm{\Delta\psi}_{L^{1+\delta}_{[0,T]} H^{\frac{3}{2}+\delta_1}_x} \le f\left(Z_0,S_0 \right) ,
        \end{equation}
        uniformly in time $T$, where $f\colon(\R_+)^2\rightarrow \R_+$ is a continuous polynomial, with $f(0,0)=0$.
    \end{lem}
    
    \begin{proof}
    We begin by rewriting ~\eqref{NLS} in a parabolic form as
    \begin{equation} \label{NLS parabolic}
        \partial_t \psi - \frac{\lambda+i}{2} \Delta\psi = -\frac{\lambda}{2} \abs{u}^2\psi - i\lambda u\cdot\nabla\psi - \mu(\lambda+i) \abs{\psi}^p \psi.
    \end{equation}
    The differential operator in this case is $A = - \frac{\lambda+i}{2} \Delta$. Comparing this to ~\eqref{differential operator}, we see that $m=1$, and $a_{\alpha} = -\frac{\lambda+i}{2}$ when $\alpha \in \{(2,0,0),(0,2,0),(0,0,2)\}$ and $a_{\alpha} = 0$ otherwise. Thus, the first condition in Definition ~\ref{K,zeta ellipticity} is satisfied. The principal symbol of the operator is
    \begin{equation*}
        \Tilde{A}(\xi) = \frac{\lambda+i}{2}\abs{\xi}^2 ,
    \end{equation*}
    from which it is clear that $\Tilde{A}(\xi)^{-1}$ is also bounded for $\abs{\xi}=1$. Finally, the spectrum $\sigma\left( \frac{\lambda+i}{2}\abs{\xi}^2 \right)$ belongs to the sector $\Sigma_{\zeta_0}$ for $\tan{\zeta_0} > \lambda^{-1} > 0$. Thus, the operator $A$ is uniformly $(K,\zeta_0)$-elliptic for $K = \max\{ \frac{\sqrt{1+\lambda^2}}{2}, \frac{2}{\sqrt{1+\lambda^2}} \}$, and the maximal parabolic regularity estimate is applicable. 

    We now act upon ~\eqref{NLS parabolic} by $(-\Delta)^{\frac{3}{4}+\frac{\delta_1}{2}}$ for some $\delta_1>0$ that will be determined shortly. We then apply Lemma ~\ref{maximal parabolic regularity} with $X=L^2(\T^3)$ and $X_1=H^2(\T^3)$, which results in
    \begin{equation} \label{applying maximal parabolic reg}
    \begin{aligned}
        &\norm{\partial_t(-\Delta)^{\frac{3}{4}+\frac{\delta_1}{2}}\psi}_{L^r_t L^2_x} + \norm{(-\Delta)^{\frac{3}{4}+\frac{\delta_1}{2}}\psi}_{L^r_t H^2_x} \\ 
        &\lesssim \norm{(-\Delta)^{\frac{3}{4}+\frac{\delta_1}{2}}\psi_0}_{(L^2_x,H^2_x)_{1-\frac{1}{r},r}} + \norm{(-\Delta)^{\frac{3}{4}+\frac{\delta_1}{2}}(\abs{u}^2 \psi)}_{L^r_t L^2_x} + \norm{(-\Delta)^{\frac{3}{4}+\frac{\delta_1}{2}}(u\cdot\nabla\psi)}_{L^r_t L^2_x} \\ 
        &\quad + \norm{(-\Delta)^{\frac{3}{4}+\frac{\delta_1}{2}}(\abs{\psi}^p \psi)}_{L^r_t L^2_x} \\
        &:= I_{17} + I_{18} + I_{19} + I_{20} ,
    \end{aligned}
    \end{equation}
	for $r=1+\delta$ with $\delta\in (0,\frac{1}{3})$, and this restriction will become clear when estimating $I_{18}$ in~\eqref{I18 estimate}. It is important to note that $L^r_t$ is calculated on $[0,T]$, where $T>0$ is the local existence time defined in~\eqref{existence time definition}. We now estimate each of the terms on the RHS. The norm of the initial condition is found to belong to a Besov space as a result of the interpolation (see ~\cite[Chapter 7]{Adams2003SobolevSpaces}). Indeed, we have
    \begin{equation} \label{I17 estimate}
    \begin{aligned}
        I_{17} &= \norm{(-\Delta)^{\frac{3}{4}+\frac{\delta_1}{2}}\psi_0}_{(L^2_x,H^2_x)_{1-\frac{1}{r},r}} = \norm{(-\Delta)^{\frac{3}{4}+\frac{\delta_1}{2}}\psi_0}_{(L^2_x,H^2_x)_{\frac{\delta}{1+\delta},1+\delta}} = \norm{(-\Delta)^{\frac{3}{4}+\frac{\delta_1}{2}}\psi_0}_{B^{\frac{2\delta}{1+\delta}}_{2,1+\delta}} \\ 
        &\lesssim \norm{(-\Delta)^{\frac{3}{4}+\frac{\delta_1}{2}}\psi_0}_{H^{\frac{2\delta}{1+\delta}+\delta_2}} \le \norm{\psi_0}_{H^{\frac{3}{2}+\frac{2\delta}{1+\delta}+\delta_1+\delta_2}} \le \norm{\psi_0}_{H^2} ,
    \end{aligned}
    \end{equation}
    for sufficiently small values for $\delta,\delta_1$, and $\delta_2$. The first inequality is due to the embedding $H^{s+\delta_2} \subset B^{s}_{2,q}$ for any $\delta_2>0$ (see ~\cite[Lemma 2.2]{Lu2021SharpSpaces}). Due to the restriction $\delta<\frac{1}{3}$, we have $\frac{2\delta}{1+\delta}<\frac{1}{2}$, implying that we may bound $I_{17}$ with the initial data, as in the last inequality of~\eqref{I17 estimate}. Next, we deal with the second term on the RHS of~\eqref{applying maximal parabolic reg} as
    \begin{equation} \label{I18 estimate}
    \begin{aligned}
        I_{18} &\lesssim \norm{\abs{u}^2 \psi}_{L^{1+\delta}_t H^2_x} \lesssim \norm{\norm{u}_{L^{\infty}_x} \norm{u}_{H^2_x} \norm{\psi}_{L^{\infty}_x}}_{L^{1+\delta}_t} + \norm{\norm{u}_{L^{\infty}_x}^2 \norm{\psi}_{H^2_x}}_{L^{1+\delta}_t} \\
        &\lesssim \norm{\norm{u}_{H^1_x}^{\frac{1}{2}} \norm{u}_{H^2_x}^{\frac{3}{2}} \norm{\psi}_{H^2_x}}_{L^{1+\delta}_t} + \norm{\norm{u}_{H^1_x} \norm{u}_{H^2_x} \norm{\psi}_{H^2_x}}_{L^{1+\delta}_t} \\
        &\lesssim \norm{u}_{L^{\frac{2(1+\delta)}{1-3\delta}}_t H^1_x}^{\frac{1}{2}} \norm{u}_{L^2_t H^2_x}^{\frac{3}{2}} \norm{\psi}_{L^{\infty}_t H^2_x} + \norm{u}_{L^{\frac{2(1+\delta)}{1-\delta}}_t H^1_x} \norm{u}_{L^2_t H^2_x} \norm{\psi}_{L^{\infty}_t H^2_x} \\
        &\lesssim \left(Z_0 + S_0^{1+\frac{2p\delta}{1+\delta}}\right)^{\frac{1}{4}} \left(Z_0 + S_0^{p+1}\right)^{\frac{3}{4}} \left(Z_0 + S_0^{\frac{p}{2}+1}\right)^{\frac{1}{2}} \\ 
        &\quad + \left( Z_0 + S_0^{1+\frac{p\delta}{1+\delta}} \right)^{\frac{1}{2}} \left(Z_0 + S_0^{p+1}\right)^{\frac{1}{2}} \left(Z_0 + S_0^{\frac{p}{2}+1}\right)^{\frac{1}{2}} .
    \end{aligned}
    \end{equation}
    The second inequality follows from the product rule for Sobolev norms~\cite[Lemma 3.4]{Majda2002VorticityFlow}, and the third inequality from Agmon's inequality and Sobolev embedding. The fourth inequality is due to H\"older's inequality, while the final step follows from ~\eqref{Z solution bound} and ~\eqref{Y dissipative estimates}. In a similar way, we can analyze the third term on the RHS of ~\eqref{applying maximal parabolic reg}, yielding
    \begin{equation} \label{I19 estimate}
        \begin{aligned}
            I_{19} &\lesssim \norm{u\cdot\nabla\psi}_{L^{1+\delta}_t H^2_x} \\
            &\lesssim \norm{\norm{u}_{L^{\infty}_x} \norm{\nabla\psi}_{H^2_x}}_{L^{1+\delta}_t} + \norm{\norm{u}_{H^2_x} \norm{\nabla\psi}_{L^{\infty}_x}}_{L^{1+\delta}_t} \\
            &\lesssim \norm{u}_{L^{\frac{2(1+\delta)}{1-3\delta}}_t H^1_x}^{\frac{1}{2}} \norm{u}_{L^2_t H^2_x}^{\frac{1}{2}} \norm{D^3\psi}_{L^2_t L^2_x} + \norm{u}_{L^2_t H^2_x} \norm{D^2\psi}_{L^{\frac{2(1+\delta)}{1-3\delta}}_t L^2_x}^{\frac{1}{2}} \norm{D^3\psi}_{L^2_t L^2_x}^{\frac{1}{2}} \\
            &\lesssim \left(Z_0 + S_0^{1+\frac{2p\delta}{1+\delta}}\right)^{\frac{1}{4}} \left(Z_0 + S_0^{p+1}\right)^{\frac{3}{4}} .
        \end{aligned}
    \end{equation}
    Since we have $p\ge 1$, the last term on the RHS of ~\eqref{applying maximal parabolic reg} is bounded using ~\eqref{superfluid mass bound} and ~\eqref{Z solution bound} as
    \begin{equation} \label{I20 estimate}
        \begin{aligned}
            I_{20} &\lesssim \norm{\abs{\psi}^p \psi}_{L^{1+\delta}_t H^2_x} \lesssim \norm{\norm{\psi}_{H^2_x}^{p+1}}_{L^{1+\delta}_t} \\
            &\lesssim \norm{\left(\norm{\psi}_{L^2_x}+\norm{\Delta\psi}_{L^2_x}\right)^{p+1}}_{L^{1+\delta}_t} \lesssim S_0^{\frac{\delta p}{2(1+\delta)}+\frac{1}{2}} + Z_0^{\frac{p+1}{2}} + S_0^{\frac{p^2}{4} + \frac{p(1+3\delta)}{4(1+\delta)} + \frac{1}{2}} .
        \end{aligned}
    \end{equation}
    Putting together the estimates in ~\eqref{I17 estimate}--\eqref{I20 estimate} gives us the desired result.
    \end{proof}

    We conclude that $\partial_t\psi, \Delta\psi \in L^{1+\delta}_{[0,T]} H^{\frac{3}{2}+\delta_1}_x$, uniformly in $T$, and that their norms can be made small by an appropriate choice of $S_0$ and $Z_0$.
	
	\subsection{Ensuring global-in-time positive density} \label{ensuring positive density}
	We have now obtained all the a priori estimates needed to return to ~\eqref{constraint to choose existence time}. 
    
    \noindent \textit{Proof of Theorem ~\ref{global existence}}: Using ~\eqref{coupling}, we have
	\begin{equation} \label{split Bpsi L^infty}
    \begin{aligned}
		\norm{B\psi}_{L^{\infty}_x} &\lesssim \norm{\Delta\psi}_{L^{\infty}_x} + \norm{\abs{u}^2\psi}_{L^{\infty}_x} + \norm{u\cdot\nabla\psi}_{L^{\infty}_x} + \norm{\abs{\psi}^p \psi}_{L^{\infty}_x} \\
        &\lesssim  \norm{D^2\psi}_{H^{\frac{3}{2}+\delta_1}_x} + \norm{u}_{H^1_x} \norm{u}_{H^2_x} \norm{\psi}_{H^2_x} + \norm{u}_{H^1_x}^{\frac{1}{2}} \norm{u}_{H^2_x}^{\frac{1}{2}} \norm{D^3\psi}_{L^2_x} + \norm{\psi}_{H^2_x}^{p+1} ,
	\end{aligned}
    \end{equation}
	where the second step is a consequence of Agmon's inequality and Sobolev embedding. We now substitute ~\eqref{split Bpsi L^infty} into the LHS of ~\eqref{constraint to choose existence time} and also use the Sobolev embedding $\norm{\psi}_{L^{\infty}_x} \lesssim \norm{\psi}_{H^2_x}$. This leads to
    \begin{equation} \label{upper bound for constraint equation LHS}
        \begin{aligned}
            \int_0^T \norm{\psi}_{L^{\infty}_x}\norm{B\psi}_{L^{\infty}_x} 
            &\lesssim \int_0^T \norm{\psi}_{H^2_x} \norm{D^2\psi}_{H^{\frac{3}{2}+\delta_1}_x} + \int_0^T \norm{u}_{H^1_x} \norm{u}_{H^2_x} \norm{\psi}_{H^2_x}^2 \\ 
            &\quad + \int_0^T \norm{\psi}_{H^2_x} \norm{u}_{H^1_x}^{\frac{1}{2}} \norm{u}_{H^2_x}^{\frac{1}{2}} \norm{D^3\psi}_{L^2_x} + \int_0^T \norm{\psi}_{H^2_x}^{p+2} \\
            &:= I_{21} + I_{22} + I_{23} + I_{24} .
        \end{aligned}
    \end{equation}
    We now show that each of these four terms can be made as small as required, by choosing sufficiently small values of the data, i.e., $S_0$ and $Z_0$. For the first term $(I_{21})$, we use H\"older's inequality,~\eqref{superfluid mass bound}, and ~\eqref{Z solution bound} to write
    \begin{equation} \label{I21 estimate p>=4}
        \begin{aligned}
            I_{21} &= \int_0^T \norm{\psi}_{H^2_x} \norm{D^2\psi}_{H^{\frac{3}{2}+\delta_1}_x} \le \norm{\psi}_{L^{1+\frac{1}{\delta}}_{[0,T]}H^2_x} \norm{D^2\psi}_{L^{1+\delta}_{[0,T]} H^{\frac{3}{2}+\delta_1}_x} \\
            &\lesssim \left(S_0^{\frac{1}{2}\left(1-\frac{p\delta}{1+\delta}\right)} + Z_0^{\frac{1}{2}} \right) \norm{D^2\psi}_{L^{1+\delta}_{[0,T]} H^{\frac{3}{2}+\delta_1}_x} .
        \end{aligned}
    \end{equation}
    The above calculation assumes that $p<1+\frac{1}{\delta}$ (which is consistent with the definition of $\delta$ in Theorem ~\ref{global existence}), so that a uniform-in-$T$ bound can be obtained. From Lemma ~\ref{lem:maximal regularity for psi}, we conclude the smallness (in terms of the initial data) of the last factor in the RHS, i.e., $\norm{D^2\psi}_{L^{1+\delta}_{[0,T]} H^{\frac{3}{2}+\delta_1}_x}$. Thus, $I_{21}$ is independent of $T$, and can be made sufficiently small by appropriate initial data.
    
    Moving on to the remaining terms in ~\eqref{upper bound for constraint equation LHS}, we have
    \begin{equation} \label{I22 estimate}
        I_{22} \lesssim \norm{u}_{L^2_t H^1_x} \norm{u}_{L^2_t H^2_x} \norm{\psi}_{L^{\infty}_t H^2_x}^2 .
    \end{equation}
    All the terms are, once again, bounded in terms of the data according to ~\eqref{Z solution bound} and ~\eqref{Y dissipative estimates}. In the same way, we have
    \begin{equation} \label{I23 estimate}
        I_{23} \lesssim \norm{\psi}_{L^{\infty}_t H^2_x} \norm{u}_{L^2_t H^1_x}^{\frac{1}{2}} \norm{u}_{L^2_t H^2_x}^{\frac{1}{2}} \norm{D^3\psi}_{L^2_t L^2_x} ,
    \end{equation}
    where all the terms are controlled by the data. Finally, using ~\eqref{superfluid mass bound} and ~\eqref{Z solution bound}, we arrive at
    \begin{equation}  \label{I24 estimate}
        \begin{aligned}
            I_{24} &\lesssim \int_0^T \left( \frac{S_0^{\frac{p}{2}+1}}{\left(1+S_0^{\frac{p}{2}}t \right)^{1+\frac{2}{p}}} +  Z_0^{\frac{p}{2}+1}e^{-\frac{t}{C}} + \frac{S_0^{\frac{(p+2)^2}{4}}}{\left(1+S_0^{\frac{p}{2}}t \right)^{\frac{(p+2)^2}{2p}}} \right) dt \\
            &\lesssim S_0 + Z_0^{\frac{p}{2}+1} + S_0^{\frac{(p+1)^2+3}{4}} ,
        \end{aligned}
    \end{equation}
    and this calculation holds for all values of $p\ge 1$. From the analysis in ~\eqref{I21 estimate p>=4}--\eqref{I24 estimate}, we conclude that the LHS of ~\eqref{upper bound for constraint equation LHS} can be made sufficiently small to satisfy the constraint given by ~\eqref{constraint to choose existence time}, for all values of time $T>0$. This implies that the density always remains bounded below, and thus solutions are global in time for $p<1+\frac{1}{\delta}$. No matter how large (but finite) $p$ is, it is possible to choose $\delta$ small enough so that the constraint in~\eqref{constraint to choose existence time} is met.
    \qed

	\addtocontents{toc}{\protect\setcounter{tocdepth}{0}}
	
    \section*{Acknowledgments}
    J.J. was supported by the NSF grants DMS-2009458 and DMS-2306910, while I.K. was supported by the NSF grant DMS-2205493. The authors appreciate the comments of the anonymous referees, which helped to improve the manuscript.
	
	\section*{Conflict of interest declaration}
	On behalf of all the authors, the corresponding author states that there is no conflict of interest.

    \section*{Data availability statement}
    Data sharing is not applicable to this article as no datasets were generated or analyzed during the current study.
	
	\addtocontents{toc}{\protect\setcounter{tocdepth}{2}}
	
	\bibliographystyle{alpha}
	\bibliography{references}
	
\end{document}